\documentclass[12pt,a4paper]{article}
\usepackage{amsfonts}
\usepackage{amsmath,amsthm}
\usepackage[all,frame]{xy}
\usepackage{enumerate}
\usepackage{multirow}
\usepackage{geometry}

\usepackage{hyperref}
\numberwithin{equation}{section}

\newtheorem{theorem}{Theorem}[section]
\newtheorem{lemma}[theorem]{Lemma}
\newtheorem{proposition}[theorem]{Proposition}
\newtheorem{corollary}[theorem]{Corollary}
\newtheorem{conjecture}[theorem]{Conjecture}
\newtheorem{remark}[theorem]{Remark}

\theoremstyle{definition}

\newcommand{\Th}{\Theta_{V}}
\newcommand{\omt}{\widetilde{\omega}}

\newcommand{\mc}[1]{\ensuremath{\mathcal{#1}}}

\newcommand{\Z}{\mathbb{Z}}

\newcommand{\C}{\mathbb{C}}
\newcommand{\HH}{\mathbb{H}}
\newcommand{\vac}{\mathbf{1}}
\newcommand{\Fim}{F^{\Omega}\left(\mathcal{D}_{\mathrm{sew}}\right)}

\DeclareMathOperator{\wt}{\mathrm{wt}\, }
\DeclareMathOperator{\Tr}{\mathrm{Tr}\, }

\DeclareMathOperator{\End}{\mathrm{End}\, }

\DeclareMathOperator{\Sp}{\mathrm{Sp}}
\newcommand{\cO}{\mathcal{O}}

\newcommand{\sups}[1]{\,^{#1} }

\begin{document}

\title{Genus Two Virasoro Correlation Functions for Vertex Operator Algebras}
\author{
Thomas Gilroy\thanks{School of Mathematics and Statistics, University College Dublin, Ireland. email: tpgilroy@gmail.com} \
and 
Michael P.~Tuite\thanks{School of Mathematics, Statistics and Applied Mathematics, NUI Galway, Ireland. email: michael.tuite@nuigalway.ie.}
}

\maketitle

\begin{abstract}
We consider all genus two correlation functions for the  Virasoro vacuum descendants of a vertex operator algebra. 
These are described in terms of explicit generating functions that can be combinatorially expressed in terms of a sequence of globally defined differential operators on which the genus two Siegel modular group $\Sp(4,\Z)$ has a natural action.
\end{abstract}

\newpage

\section{Introduction}
\label{Intro}
A Vertex Operator Algebra (VOA) (e.g. \cite{FLM}, \cite{K}, \cite{LL}, \cite{MT1}) is an algebraic  system related to Conformal Field Theory (CFT) in physics e.g. \cite{DMS}. An essential ingredient of a VOA or  CFT is the existence of a conformal Virasoro vector whose vertex operator modes generate the  Virasoro algebra of central charge $c$. The connection between VOAs, the Virasoro vector and genus one Riemann surfaces was established in the foundational work of Zhu \cite{Z} giving a rigorous basis for many ideas in CFT. 
Zhu established a general recursion formula  relating any genus one $n$-point correlation function to linear combination of  $(n-1)$-point correlation functions. When applied to genus one Virasoro correlation functions,  Zhu recursion implies the genus one Ward identities. For a suitable VOA (e.g. if $V$ is $C_2$ cofinite \cite{Z}) the Ward identities imply the partition function for $V$ and its ordinary modules satisfy a genus one modular ordinary differential equation. This paper lays some of the ground work for the development of a new theory of genus two $\Sp(4,\Z)$ Siegel modular partial differential equation for the genus two partition function of a suitable VOA $V$ and its ordinary modules.  
This is illustrated in a sequel paper  where we describe  a new $\Sp(4,\Z)$ modular partial differential equation for the partition functions (for $V$ and its ordinary  modules) for the $(2,5)$  minimal Virasoro model of central charge $c=-\frac{22}{5}$  \cite{GT2}.  

\medskip

Correlation functions for VOAs on a genus two Riemann surface have been defined, and in some cases calculated based on an explicit sewing procedure for sewing two tori together \cite{MT2,MT3}.
We recently  described a general Zhu recursion formula for genus two $n$-point correlation functions  which gives rise to new genus two Ward identities when applied to Virasoro vector $n$-point functions \cite{GT1}.
The purpose of this paper is to describe all genus two correlation functions for Virasoro descendants of the vacuum vector in terms of explicit generating functions. These generating functions, which satisfy the genus two Ward identity, are shown to be combinatorially expressed in terms of a  sequence of globally defined differential operators on which the Siegel modular group $\Sp(4,\Z)$  has an natural action.
 
 \medskip

We begin in Section \ref{Review} with a very brief review of VOAs on a genus two Riemann surface. We review a sewing scheme for constructing a genus two surface from two punctured tori and how to formally define the genus two partition and $n$-point correlation functions in terms of genus one VOA data \cite{MT2,MT3}.  

In Section \ref{genus2chap} contains the main results of this paper. We first show that genus two $n$-point correlation functions for $n$ Virasoro vectors  are  generating functions for the correlation functions for all Virasoro vacuum descendants in a similar fashion to the genus zero and one cases \cite{HT1}. 
These generating functions satisfy  genus two Ward identities (derived from genus two Zhu recursion) which involves genus two \emph{generalised Weierstrass functions} related to a global  (2,−1)-bidifferential $\Psi(x,y)$ holomorphic for $x\neq y$ \cite{GT1}. 
We also describe some analytic differential equations, which also involve $\Psi(x,y)$, for the genus two bidifferential $\omega(x,y)$, normalised holomorphic 1-differentials $\nu_{1}(x),\nu_{2}(x)$ and the projective connection $s(x)$.
Using these differential equations, we demonstrate in Theorem \ref{theor:main} how to express each  generating function  in a symmetric way as a sum  of given weights of appropriate graphs. 
In particular, the Virasoro vector $n$-point function is determined  by the action of a specific    symmetric differential operator $\cO_n$ on the normalised partition function
$Z_V^{(2)}/ (Z_M^{(2)} )^{c}$ (cf. \eqref{eq:ThetaV}) where $Z_V^{(2)}$ denotes the genus two partition function for $V$,  $Z_M^{(2)}$ denotes the genus two partition function for the Heisenberg VOA $M$ and $c$ is the central charge.

Lastly in Section~\ref{sect:an_mod} we consider general analytic and modular transformation properties of the differential operator $\cO_n$  in any coordinate system on an arbitrary genus two Riemann surface. In particular, in Theorem~\ref{th:OnMod}, we show how $\cO_n$ transforms 
under the Siegel modular group $\Sp(4,\Z)$ having established that 
the global  (2,−1)-bidifferential $\Psi(x,y)$ is  $\Sp(4,\Z)$ invariant in Theorem~\ref{th:2P1mod}.

\section{Vertex operator algebras on a genus two Riemann surface}\label{Review}
\subsection{Genus two Riemann surfaces}\label{ssect:g2}
We briefly review some  concepts in genus two Riemann surface theory e.g. \cite{FK,F,Mu1}. 
Let  $\mathcal{S}^{(2)}$ be a compact genus two Riemann surface with canonical homology basis $\alpha^{i},\beta^i$ for $i=1, 2$. 
There exists a unique  holomorphic symmetric bidifferential $(1,1)$-form $\omega(x,y)$, the \emph{normalised bidifferential of the second kind},  where
for  $x\neq y\in \mathcal{S}^{(2)}$ 
\begin{align}\label{omega}
\omega(x,y)&=\frac{dxdy}{(x-y)^{2}}+\frac{1}{6}s(x)-\frac{(x-y)}{12}\partial_x s(x)+O\left((x-y)^2\right),
\\
\oint_{\alpha^{i}}\omega(x,\cdot )&=0,\quad i=1,2.
\notag
\end{align} 
$s(x)$ is the  \emph{projective connection}  transforming under an analytic map 
$x\rightarrow \phi(x)$ as
\begin{align}
s(x)=s(\phi(x)) +\{\phi(x),x\}dx^2,
\label{eq:sphi}
\end{align}
 where  $\{\phi(x),x\}=\frac{\phi'''(x)}{\phi'(x)}-\frac{3}{2}\left(\frac{\phi''(x)}{\phi'(x)} \right)^2$ is the Schwarzian derivative.
Futhermore
\begin{align*}
\nu_{i}(x)=\oint_{\beta^{i}}\omega(x,\cdot ), \qquad 
 \Omega_{ij}=\frac{1}{2\pi i}\oint_{\beta^{i}}\nu_{j}, \quad i,j=1,2.
\end{align*}
for  \emph{holomorphic differentials} $\nu_{i}(x)$ normalised by $\oint_{\beta^{i}}\nu_{j}=2\pi i\delta_{ij}$ 
and  \emph{period matrix} $\Omega  \in \HH_{2}$, the genus two Siegel upper half plane i.e. $\Omega=\Omega^T$ and $\Im(\Omega)>0$.

\medskip
Consider the genus two Riemann surface $\mathcal{S}^{(2)}$ constructed by sewing two genus one tori $\mathcal{S}_{a}=\C/\Lambda_a $, for lattice $\Lambda_a=2\pi i(\Z\tau_a\oplus\Z)$  with modular parameter $\tau_a\in \HH_1$ for $a=1,2$ \cite{MT2}. 
Let $z_a\in \mathcal{S}_{a}$, $\epsilon\in \C$ and define punctured tori 
\begin{align*}
\widehat{\mathcal{S}}_{1}=
\mathcal{S}_{1}\backslash\left \{z_{1},\left\vert z_{1}\right\vert \leq |\epsilon |/r_{2}\right\},
\qquad
\widehat{\mathcal{S}}_{2}=
\mathcal{S}_{2}\backslash\left \{z_{2},\left\vert z_{2}\right\vert \leq |\epsilon |/r_{1}\right\},
\end{align*}
where $|\epsilon |\leq r_{1}r_{2}$. 
We identify the annular regions  
$\{z_{1},|\epsilon |/r_{2}\leq \left\vert
z_{1}\right\vert \leq r_{1}\}$ and 
$\{z_{2},|\epsilon |/r_{1}\leq \left\vert
z_{2}\right\vert \leq r_{2}\}$
via  the sewing relation 
$z_{1}z_{2}=\epsilon$.
Then $\mathcal{S}^{(2)}$ is parameterized by the sewing domain 
\begin{align}\label{Deps}
\mathcal{D}_{\mathrm{sew}}=\left\{(\tau _{1},\tau _{2},\epsilon )\in \HH_{1}
\times\mathbb{ H}_{1}\times\mathbb{ C}: |\epsilon |<\frac{1}{4}
D(q_{1})D(q_{2})\right\},
\end{align}
where $q_{a}=e^{2\pi i\tau_a}$ and 
$D(q_a )=\min_{\lambda_a\in\Lambda_a,\lambda_a\neq 0}|\lambda_a |$. We may then obtain explicit expressions for $\omega(x,y)$, $\nu_i(x)$ and $\Omega_{ij}$  on $\mathcal{S}^{(2)}$ for $x,y\in \widehat{\mathcal{S}}_{1} \cup \widehat{\mathcal{S}}_{2}$ described in \cite{MT2}.

\subsection{Vertex operator algebras on a torus}
We review  aspects of vertex operator algebras (e.g. \cite{ FLM,K,LL,MT1}). 
A Vertex Operator Algebra (VOA) is a quadruple $(V,Y,\mathbf{1},\omega)$ consisting of a $\Z$-graded complex vector space $V = \bigoplus_{n\in\Z}\,V_{(n)}$ where $\dim V_{(n)}<\infty$ for each $n\in \Z$, a linear map $Y:V\rightarrow\End(V)[[z,z^{-1}]]$ for a formal parameter $z$ and pair of distinguished vectors: the vacuum $\mathbf{1}\in V_{(0)}$ and the conformal vector $\omega\in V_{(2)}$. For each $v\in V$, the image under the map $Y$ is the \emph{vertex operator}
\begin{align*}
Y(v,z) = \sum_{n\in\Z}v(n)z^{-n-1},
\end{align*}
with \emph{modes} $v(n)\in\End(V)$, where $Y(v,z)\mathbf{1} = v+O(z)$. Vertex operators satisfy \emph{locality} i.e. for all $u,v\in V$ there exists an integer $k\geq 0$ such that
\begin{align*}
(z_1 - z_2)^k \left[ Y(u,z_1), Y(v,z_2) \right] = 0.
\end{align*}
The vertex operator of the conformal vector $\omega$ is 
$Y(\omega,z) = \sum_{n\in\Z}L(n)z^{-n-2}$
where the modes $L(n)$ satisfy the Virasoro algebra with \emph{central charge} $c$
\begin{align*}
&[L(m),L(n)] = (m-n)L(m+n) + c\,\frac{m^3-m}{12}\delta_{m,-n}\,\mathrm{Id}_V.
\end{align*}
Furthermore, $L(0)v = kv$ for\emph{ conformal weight} $\wt(v)=k$ for all $v\in V_{(k)}$ and $Y(L(-1)u,z) = \partial_z Y(u,z)$.

\bigskip 
In order to describe VOAs on a torus, Zhu \cite{Z} introduced an isomorphic VOA $(V,Y[~,~],\mathbf{1},\omt)$ with ``square bracket'' vertex operators
\begin{align*}
Y[v,z] = \sum_{n\in \Z} v[n] z^{-n-1} = Y\left(e^{{L(0)}}v, e^{z}-1\right),
\end{align*} 
and conformal vector $
\omt = \omega - \frac{c}{24}\mathbf{1}
$
 with Virasoro modes $ L[n] $.

\medskip

Define the  \emph{genus one partition function} by the formal trace $Z^{(1)}_V(\tau) = \Tr_V\left( q^{L(0)-c/24} \right)$, with $q=e^{2\pi i\tau}$, the \emph{genus one correlation one point function} by the formal trace 
\begin{align*}
Z^{(1)}_V(v;\tau) = \Tr_V\left( o(v) q^{L(0)-c/24} \right),\quad v\in V,
\end{align*}
where $o(v)=v(k-1)$ for  $v\in V_{(k)}$ 
and the \emph{genus one $n$-point correlation function} for 
$v_1, \ldots,v_n  \in V$ inserted at $z_1, \ldots,z_n\in \C/(2\pi i(\Z\tau\oplus\Z))$ by
\begin{align*}
Z^{(1)}_V(v_1,z_1;\ldots;v_n,z_n;\tau) =Z^{(1)}_V(Y[v_1,z_1]\ldots Y[v_n,z_n]\vac;\tau). 
\end{align*}
Zhu describes a general recursion formula for expressing any  genus one $n$-point correlation function  as  a linear combination  of $(n-1)$-point functions with universal coefficients given by explicit  elliptic functions \cite{Z}. 
All of the above definitions can be naturally extended to define $Z^{(1)}_{W}(\ldots)$ for an ordinary graded $V$-module $W$ where  the trace is taken over $W$. 

\subsection{VOAs on a genus two Riemann surface}
We define the genus two partition function and $n$-point correlation function for a VOA based on the sewing scheme for $\mathcal{S}^{(2)}$ constructed from  two tori $\mathcal{S}_1$ and $\mathcal{S}_2$  \cite{MT3,GT1}.  
The \emph{genus two partition function} for  $V$ of strong CFT-type is defined by 
\begin{align}
Z^{(2)}_V(\tau_1,\tau_2,\epsilon) = \sum_{u\in V} Z^{(1)}_V(u;\tau_1) Z^{(1)}_V(\overline{u};\tau_2),
\label{eq:Zdef}
\end{align}
where the formal sum is taken over any $V$-basis and $\overline{u}$ is the dual of $u$ with respect to an invariant invertible bilinear form $\langle~,~\rangle$ associated with the Mobius map $z\rightarrow \epsilon/z$ (see \cite{GT1} for more details).

The \emph{genus two} $n$\emph{-point correlation function} for $a_1,\ldots,a_L\in V$ 
and $b_1,\ldots,b_R\in V$ formally  inserted at $x_1,\ldots,x_L\in\widehat{\mathcal{S}}_1$ and $y_1,\ldots,y_R\in\widehat{\mathcal{S}}_2$, respectively, is defined by
\begin{align}
&Z^{(2)}_V(a_1,x_1;\ldots;a_L,x_L|b_1,y_1;\ldots;b_R,y_R;\tau_1,\tau_2,\epsilon) 
\notag
\\
&= \sum_{u\in V} Z^{(1)}_V(Y[a_1,x_1]\ldots Y[a_L,x_L]u;\tau_1) Z^{(1)}_V(Y[b_R,y_R]\ldots Y[b_1,y_1]\overline{u};\tau_2).
\label{eq:Znptdef}
\end{align}
Convergent expressions have been found for such correlation functions for  particular VOAs such as the Heisenberg VOA and  lattice VOAs \cite{MT3}. 
A formal Zhu recursion formula for a genus two  $n$-point function   in terms of   $(n-1)$-point functions  is described in \cite{GT1} where the coefficients  are formal series, called \emph{generalised Weierstrass functions}, which depend on the conformal weight of the recursion vector but are otherwise universal. 
For  conformal weight $1$ or $2$, these series are holomorphic on appropriate domains \cite{GT1}. 
The above definition  can be naturally extended to define $Z^{(2)}_{W_1,W_2}(\ldots)$ for a pair of ordinary graded $V$-modules $W_1,W_2$ where  the left or right hand   trace is taken over $W_1$ or $W_2$ respectively.

\section{Genus Two Virasoro Correlation Functions}
\label{genus2chap}
\subsection{The  Virasoro generating function}
Consider the genus two Virasoro $n$-point correlation function 
for  $z_1,\ldots,z_n\in \widehat{\mathcal {S}}_1$
\begin{align}
Z_V^{(2)}\left(\omt,z_1;\ldots;\omt,z_n\right)=
\sum_{u\in V} Z^{(1)}_V\left (Y[\omt,z_1]\ldots Y[\omt,z_n]u;\tau_1\right ) Z^{(1)}_V(\overline{u};\tau_2),
\label{eq:Virnpt}
\end{align}
(where we suppress the dependence on $\tau_1,\tau_2,\epsilon$).
We define the formal differential
\begin{align}
G_n(\mathbf{z})=G_n(z_1,\ldots,z_n)=Z_V^{(2)}\left(\omt,z_1;\ldots;\omt,z_n\right)d\mathbf{z}^2,
\label{eq:Gm}
\end{align}
where  $d\mathbf{z}^2=dz_i^2\ldots dz_n^2$. 
$G_n(\mathbf{z})$  
 is independent of whether we formally  insert  $\omt$ at $z_i\in \widehat{\mathcal{S}}_1$ or at $ \epsilon/z_i\in \widehat{\mathcal{S}}_2$ (Proposition~6 of \cite{MT3}).
Similarly to \cite{HT1} we find
\begin{proposition}
\label{prop_Ggen}
$G_m(\mathbf{z})$ is symmetric in $z_i$ and is a  generating function for all genus two $n$-point correlation functions for Virasoro vacuum descendants.
\end{proposition}

\begin {proof}
$G_m(\mathbf{z})$ is a symmetric  in $z_1,\ldots,z_m$ by locality.  
Consider the genus two $n$-point function for $n$ Virasoro vacuum descendants
 $v_i=L[-k_{i1}]\ldots L[-k_{im_i}]\vac$ inserted at $z_i\in \widehat{\mathcal{S}}_1$ for $i=1,\ldots,n$ and $k_{ij}\ge 2$ given by
\begin{align*}
Z_V^{(2)}\left(v_{1},z_1; \ldots; v_{n},z_n\right) 
=\sum_{u\in V} Z_V^{(1)}\left(Y[v_{1},z_1]\ldots Y[v_{n},z_n]u;\tau_1\right) Z^{(1)}_V(\overline{u};\tau_2).
\end{align*}  
$Z_V^{(1)}\left(Y[v_{1},z_1]\ldots Y[v_{n},z_n]u;\tau_1\right)$  is the coefficient of $\prod_{i=1}^{n}\prod_{j=1}^{m_{i}}(x_{ij})^{k_{ij}-2}$  in 
\[
Z_V^{(1)}\left(
Y[Y[\omt,x_{1}]\ldots Y[\omt,x_{1m_1} ]\vac,z_1]\ldots 
Y[Y[\omt,x_{n1}]\ldots Y[\omt,x_{nm_n} ]\vac,z_n]u;\tau_1\right).
\] 
Using   associativity and lower truncation  (e.g. \cite{K,LL,MT1}) we find for $N\gg 0$ that
\begin{align*}
&\prod_{i=1}^{n}\prod_{j=1}^{m_{i}}(x_{ij}+z_i)^{N}
Y[Y[\omt,x_{11}]\ldots Y[\omt,x_{1m_1} ]\vac,z_1]\ldots 
Y[Y[\omt,x_{n1}]\ldots Y[\omt,x_{nm_n} ]\vac,z_n]u
\\
&=\prod_{i=1}^{n}\prod_{j=1}^{m_{i}}(x_{ij}+z_i)^{N}
Y[\omt,z_1+x_{11}]\ldots Y[\omt,z_1+x_{1m_1} ]\ldots 
Y[\omt,z_n+x_{n1}]\ldots Y[\omt,z_n+x_{nm_n}]u.
\end{align*}
 Thus the genus two $n$-point function for $v_1,\ldots ,v_n$
is the coefficient of $\prod_{i=1}^{n}\prod_{j=1}^{m_{i}}(x_{ij})^{k_{ij}-2}$ 
of the formal expansion of $Z_V^{(2)}\left(\omt,z_1+x_{11};\ldots;\omt,z_n+x_{nm_n}\right)$ for $M=\sum_{i=1}^n m_{i}$.
\end{proof}

\subsection{A genus two Ward identity}
Define a genus two modular derivative operator 
\begin{align}
\nabla_x=\sum_{1\le a\le b\le 2} \nu_{a}(x)\nu_{b}(x)\frac{\partial}{\partial \Omega_{ab}},
\label{eq:nabla}
\end{align}
for period matrix $\Omega_{ab}$, and normalised holomorphic 1-differentials $\nu_{a}$.
There exists an injective but non-surjective holomorphic map $F^{\Omega }$ from the sewing domain  $\mathcal{D}_{\mathrm{sew}}$  into the Siegel upper half plane \cite{GT1}
\begin{align}
F^{\Omega }:\mathcal{D}_{\mathrm{sew}} &\rightarrow \mathbb{H}_{2}, \notag
\\
(\tau _{1},\tau _{2},\epsilon ) &\mapsto \Omega (\tau _{1},\tau_2,\epsilon),\label{eq:FOm}
\end{align}
Below we will also denote by  $\nabla_x$ the action of $\left(F^{\Omega}\right)^{-1} \circ \nabla_x \circ F^{\Omega}$ on $\mathcal{D}_{\mathrm{sew}}$.

In Section~5 of \cite{GT1} we describe a genus two Ward identity for  the genus two Virasoro $n$-point correlation  function. This is expressed in terms of generalised Weierstrass functions
$\sups{2}\mathcal{P}_{k}(x,y)$ for $k\ge 1$ defined as follows. 
Let $\boldsymbol{\nu}(x)=\left[\nu_1(x),\nu_2(x)\right] $
denote a row vector of holomorphic 1-differentials and define 
\begin{align}
\Psi(x,y)=
\sups{2}\mathcal{P}_{1}(x,y)dx^2(dy)^{-1}=
-\;\frac{
\omega(x,y)\begin{vmatrix}
\boldsymbol{\nu}(x) \\
\boldsymbol{\nu}(y)
\end{vmatrix}
+\begin{vmatrix}
\boldsymbol{\nu}(y)\\
\nabla_x\boldsymbol{\nu}(y)
\end{vmatrix}
}
{
\begin{vmatrix}
\boldsymbol{\nu}(y) \\
 \partial_{y} \boldsymbol{\nu}(y)
\end{vmatrix}
dy
}.
\label{eq:P21form}
\end{align}
%
\begin{proposition}[\cite{GT1}]\label{prop:P21}
$\Psi(x,y)$ is a holomorphic  $(2,-1)$-bidifferential for $x\neq y$ 
where, for any local coordinates $x,y$
\begin{align*}
\Psi(x,y)= \left(\frac{1}{x-y}+\text{regular terms}\right)dx^2(dy)^{-1}.
\end{align*}
\end{proposition}
\noindent We also define generalised Weierstrass functions  for $k\ge 1$ by
\begin{align}
\sups{2}\mathcal{P}_{k}(x,y)=\frac{1}{(k-1)!}\partial_{y}^{k-1}\left(\sups{2}\mathcal{P}_{1}(x,y)\,\right)=\frac{1}{(x-y)^k}+\mbox{regular terms},
\label{eq:P2kform}
\end{align}
which is holomorphic for $x\neq y$.  
\begin{proposition}[\cite{GT1}]\label{prop:Vnpt}
$G_n(\mathbf{z})$  obeys the formal   Ward identity
for  $z_1,\ldots,z_n\in \widehat{\mathcal {S}}_1 \cup  \widehat{\mathcal {S}}_2$
and $(\tau _{1},\tau _{2},\epsilon )\in \mathcal{D}_{\mathrm{sew}} $
\begin{align}
G_n(\mathbf{z})
=& \Bigg( \nabla_{z_1} 
+ dz_1^2\sum_{k=2}^n \left( \sups{2}\mathcal{P}_{1}(z_1,z_k) \partial_{z_k} + 2\cdot\sups{2}\mathcal{P}_{2}(z_1,z_k) \right)\Bigg){G}_{n-1}(z_2,\ldots,z_n) 
\notag
\\
&+ \frac{c}{2} \sum_{k=2}^n \sups{2}\mathcal{P}_{4}(z_1,z_k) {G}_{n-2}(z_2,\ldots,\widehat{z}_k,\ldots, z_n)\, dz_1^2dz_k^2,
\label{eq:Ward}
\end{align}
where $\widehat{z}_k$  denotes the omission of the given term. 
\end{proposition}
We note that the above expression for $G_n(\mathbf{z})$ is not manifestly symmetric in its arguments.

\subsection{Some analytic differential equations}
For the genus two bidifferential $\omega(x,y)$, normalised holomorphic 1-differentials $\nu_{a}(x)$ for $a=1,2$ and the projective connection $s(x)$ we find from Section~6 of \cite{GT1} that 
\begin{proposition}\label{diffeqn}
$\omega(x,y)$, $\nu_{a}(x)$ for a=1,2 and  $s(x)$  satisfy the following analytic differential equations 
\begin{align}
&\Big( \nabla_x + dx^2 \sum_{r=1}^2 \left( \sups{2}\mathcal{P}_{1}(x,y_r)\partial_{y_r} + \sups{2}\mathcal{P}_{2}(x,y_r) \right) \Big) \omega(y_1,y_2)
=  \omega(x,y_1)\omega(x,y_2),
\label{eq:omDE}
\\
&\Big(\nabla_x +  dx^2 \left(\sups{2}\mathcal{P}_{1}(x,y) \partial_y +\sups{2}\mathcal{P}_{2}(x,y) \right)\Big) \nu_a(y)
=
\omega(x,y) \nu_a(x), 
\label{eq:nuDE}
\\
&\Big(\nabla_x +  dx^2 \left(\sups{2}\mathcal{P}_{1}(x,y) \partial_y +2\sups{2}\mathcal{P}_{2}(x,y) \right)\Big)  s(y) + 6\sups{2}\mathcal{P}_{4}(x,y)dx^2dy^2
=
6\,\omega(x,y) ^2, 
\label{eq:SDE}
\end{align}
for all $x,y_1,y_2\in \widehat{\mathcal{S}}_1 \cup \widehat{\mathcal{S}}_2 $  and $\Omega\in\Fim$.
\end{proposition}
We may generalise \eqref{eq:omDE} and \eqref{eq:nuDE} in the following coordinate independent way:
\begin{corollary}\label{diffeqn_indep}
$\omega(x,y)$ and  $\nu_{a}(x)$ for a=1,2  satisfy the following coordinate independent analytic differential equations for all  $\Omega\in\HH_2$
\begin{align}
\nabla_x \,\omega(y_1,y_2) +  \sum_{r=1}^2 \partial_{y_r}\left( \Psi(x,y_r) \omega(y_1,y_2)\right) dy_r
&=  \omega(x,y_1)\omega(x,y_2),
\label{eq:omDE_indep}
\\
 \nabla_x \,\nu_a(y)+   \partial_y\left( \Psi(x,y) \nu_a(y)\right) dy
&=
\omega(x,y) \nu_a(x). 
\label{eq:nuDE_indep}
\end{align}
\end{corollary}
\begin{proof}
$\Psi(x,y_1) \omega(y_1,y_2)$ is a global $(2,0,1)$-form in $(x,y_1,y_2)$ (with a similar statement for $\Psi(x,y_2) \omega(y_1,y_2)$). Thus 
\begin{align*}
 d_{y_r}\left( \Psi(x,y_r) \omega(y_1,y_2)\right)=\partial_{y_r}\left( \Psi(x,y_r) \omega(y_1,y_2)\right) dy_r,
\end{align*}
is a global  $(2,1,1)$-form. Hence \eqref{eq:omDE} can be expressed by \eqref{eq:omDE_indep} in a coordinate independent way for all $\Omega\in\Fim$.  But all parts of \eqref{eq:omDE_indep} are holomorphic for all $\Omega\in \HH_2$ and hence the identity can be analytically extended from $ \Fim$ to $\HH_2$.
\eqref{eq:nuDE_indep} follows from \eqref{eq:omDE_indep} by integrating $y_2$ along the $\beta^a$ homology cycle.
\end{proof}
We may also generalise \eqref{eq:SDE} in the following way:
\begin{corollary}\label{Seqn_indep}
$s(y)$, for a given choice of local coordinate $y$, satisfies the following  analytic differential equation  for all  $\Omega\in\HH_2$
\begin{align}
&\Big(\nabla_x +  dy \left(\Psi(x,y) \partial_y +2 \,\partial_y\Psi(x,y) \right)\Big) s(y)+ dy^3\,\partial_y^3\Psi(x,y)
=
6\,\omega(x,y)^2.
\label{eq:SDE_indep}
\end{align}
\end{corollary}
\begin{proof}
We let $y_2=y$ and $y_1=y+\varepsilon$ and note from \eqref{omega} that 
\begin{align*}
\omega(y_1,y_2)&=\frac{dy^2}{\varepsilon^2}+\frac{1}{6}s(y)+O(\varepsilon),
\\
\partial_{y_1}\omega(y_1,y_2)&=-\frac{2dy^2}{\varepsilon^3}+\frac{1}{12}\partial_y s(y)+O(\varepsilon),
\\
\partial_{y_2}\omega(y_1,y_2)&=\frac{2dy^2}{\varepsilon^3}+\frac{1}{12}\partial_y s(y)+O(\varepsilon),
\\
\Psi(x,y_1)&=\Psi(x,y)+\varepsilon\partial_y\Psi(x,y) 
+\frac{\varepsilon^2}{2}\partial_y^2\Psi(x,y)+\frac{\varepsilon^3}{6}\partial_y^3\Psi(x,y)+O(\varepsilon^4),
\\
\partial_y\Psi(x,y_1)&=\partial_y\Psi(x,y)+\varepsilon\partial_y^2\Psi(x,y)+\frac{\varepsilon^2}{2}\partial_y^3\Psi(x,y)+O(\varepsilon^3).
\end{align*}
The result follows by substituting the above into \eqref{eq:omDE_indep} and taking
the $\varepsilon\rightarrow 0$ limit. 
\end{proof}

We finally note that the genus two partition function $Z_M^{(2)}(\tau_1,\tau_2,\epsilon)$  for the Heisenberg VOA $M$ obeys \cite{GT1}
\begin{proposition}\label{prop:ZM}
$Z_M^{(2)}(\tau_1,\tau_2,\epsilon)$ is holomorphic for $(\tau_1,\tau_2,\epsilon)\in \mathcal{D}_{\mathrm{sew}}$ and satisfies 
\begin{align}
\nabla_x Z_M^{(2)}=\frac{1}{12}s(x)Z_M^{(2)},
\label{eq:DZM}
\end{align}
for $x \in \widehat{\mathcal{S}}_1 \cup \widehat{\mathcal{S}}_2 $. 
\end{proposition}
\begin{remark}
\label{rem:ZM}
$Z_M^{(2)}(\tau_1,\tau_2,\epsilon)$ can be considered as a holomorphic function on $\Fim$ but \textbf{cannot} be analytically continued to the full Siegel upper half plane $\HH_2$ (cf. Theorem~7.2, \cite{GT1}). In physics, this follows from the conformal anomaly (e.g. \cite{FS}) which for the Heisenberg VOA is believed to be related to the non-existence of a global section of certain determinant line bundles on the genus two Riemann surface \cite{Mu2}.
\end{remark}
\eqref{eq:omDE}-\eqref{eq:DZM} are the genus two analogues of differential equations for elliptic and modular functions described in \cite{HT1}. Thus \eqref{eq:DZM} corresponds to
\begin{align*}
q\frac{\partial}{\partial q}\left(\frac{1}{ \eta(q)}\right)=\frac{1}{2}E_2(q) \left(\frac{1}{ \eta(q)}\right),
\end{align*}
for the weight 2 quasi-modular Eisenstein series $E_2(q) =-\frac{1}{12}+2\sum_{m,n\ge 1}nq^{mn}$.

\medskip

\subsection{The main theorem}\label{ssect:graphs}
We show below in Theorem~\ref{theor:main} how to express $G_n(\mathbf{z})$ in a manifestly symmetric fashion as a sum of weights of appropriate graphs. 
The graph configurations are precisely those exploited in \cite{HT1} to describe genus one Virasoro $n$-point functions and many of the arguments below mirror the genus one case. However, the graph weights are differently defined in the genus two case and the technicalities are more involved. 
Furthermore, the genus two graph weights for $G_n(\mathbf{z})$ are described in terms of a linear  differential operator $\cO_n(\mathbf{z})$ which is symmetric in its arguments and  possesses fundamental properties under analytic and genus two $\Sp(4,\Z)$ modular transformations. 

Define, for central charge $c$, a formal normalised partition function
\begin{align}
\Th(\tau_1,\tau_2,\epsilon): &= Z_M^{(2)}(\tau_1,\tau_2,\epsilon)^{-c}Z_V^{(2)}(\tau_1,\tau_2,\epsilon),
\label{eq:ThetaV}
\end{align}
where $Z_M^{(2)}(\tau_1,\tau_2,\epsilon)$ is the genus two partition function for the Heisenberg VOA (which is holomorphic on $\mathcal{D}_{\mathrm{sew}}$). Following Remark~\ref{rem:ZM} and \cite{FS} we conjecture:
\begin{conjecture}\label{conj:ThV}
$\Th$ is holomorphic on $\HH_2$ for  a  $C_2$-cofinite VOA $V$. If $V$ is also rational then $\Th$
is a component of a vector valued Siegel modular form of weight $c/2$. 
\end{conjecture}
\noindent For example, for a lattice  VOA $V_L$ for an even lattice $L$ of rank $c$ we find $\Th=\Theta_{L}(\Omega)$, the genus two Siegel lattice theta function \cite{MT3}, a Siegel modular form of weight $c/2$ and level $|L^*/L|$ where $L^*$ is the  dual lattice. 
Thus, conjecturally,  it is the normalised partition function $\Th$ on which the full genus two $\Sp(4,\Z)$ modular group naturally acts.\footnote{We also note that $\Th$ has canonical properties in the one torus degeneration limite \cite{HT2}.}
One of the main purposes of this paper is to develop global differential operators $\cO_n(\mathbf{z})$ on which $\Sp(4,\Z)$ acts.
In the sequel \cite{GT2} we show how these operators give rise to a $\Sp(4,\Z)$ modular differential equation for the partition function for the $(2,5)$  minimal Virasoro model of central charge $-22/5$. 

\medskip

We define the linear differential operator $\cO_n(\mathbf{z})$ (which in general acts on differentiable  functions of $\Omega$) by
\begin{align}
\cO_n(\mathbf{z})\Theta_V:&=
Z_M^{(2)}(\tau_1,\tau_2,\epsilon)^{-c} {G}_n(\mathbf{z}).
\label{eq:On}
\end{align}

For $n=1$ we find $G_1(z_{1})=\nabla_{z_{1}} Z_{V}^{(2)}(\tau_1,\tau_2,\epsilon)$
so that, using \eqref{eq:DZM}, we find\footnote{
The operator \eqref{eq:On1} is a like a higher genus Serre derivative as discussed further in \cite{GT1}.}
\begin{align}
\mathcal{O}_1(z_{1})
=\nabla_{z_{1}}+ \frac{c}{12}s(z_{1}).
\label{eq:On1}
\end{align}

In order to describe the $n=2$ case, we also define the differential operator
\begin{align}
\mc{D}_{z_1,z_2}=\nabla_{z_1} 
+ dz_1^2 \left( \sups{2}\mathcal{P}_{1}(z_1,z_2) \partial_{z_2} + 2\,\sups{2}\mathcal{P}_{2}(z_1,z_2)\right).
\label{eq:Dop}
\end{align}
Then for $n=2$, the Ward identity \eqref{eq:Ward} implies
\begin{align*}
\mathcal{O}_2(z_1,z_2)\Theta_V=&\left(Z_M^{(2)}\right)^{-c} 
\mc{D}_{z_1,z_2}\nabla_{z_2} Z_{V}^{(2)} 
+ \frac{c}{2} \sups{2}\mathcal{P}_{4}(z_1,z_2) \,\Th\, dz_1^2dz_2^2\\
 =&\,\mc{D}_{z_1,z_2}\left(\nabla_{z_2}\Th+ \frac{c}{12}s(z_2)\Th\right)+\frac{c}{12}s(z_1)\nabla_{z_2}\Th
\notag
\\
&+ \frac{c^2}{144}s(z_1)s(z_2)\Th+ \frac{c}{2} \sups{2}\mathcal{P}_{4}(z_1,z_2) \, dz_1^2dz_2^2\,\Th.
\end{align*}
From \eqref{eq:nuDE} we note that
\begin{align}
\mc{D}_{z_1,z_2}\nu_a(z_2)\nu_b(z_2)
=\omega(z_1,z_2)\left(\nu_a(z_1)\nu_b(z_2)+\nu_a(z_2)\nu_b(z_1)\right).
\label{eq:Dnunu}
\end{align}
\eqref{eq:Dnunu} together with \eqref{eq:SDE} imply that 
\begin{align}
\cO_{2}(z_1,z_2)=&\sum_{1\le a\le b\le 2} \sum_{1\le c\le d\le 2} 
\nu_{a}(z_1)\nu_{b}(z_1)\nu_{c}(z_2)\nu_{d}(z_2) 
\frac{\partial^2}{\partial \Omega_{ab} \partial \Omega_{cd}}
\notag\\
&
+ \frac{c}{12}s(z_1)\nabla_{z_2}+ \frac{c}{12}s(z_2)\nabla_{z_1}+ \frac{c^2}{144}s(z_1)s(z_2)
\notag\\
&+ 2\, \omega(z_1,z_2)\sum_{1\le a\le b\le 2} \nu_{a}(z_1)\nu_{b}(z_2)
\frac{\partial}{\partial \Omega_{ab}} + \frac{c}{2}\omega(z_{1},z_{2})^2.
\label{eq:On2}
\end{align}
This expression is clearly symmetric in $z_1,z_2$ in accordance with 
Proposition~\ref{prop_Ggen}. Furthermore, each term in \eqref{eq:On2} is now written in coordinate independent way. 
\medskip 

Similarly to Section~3 of \cite{HT1} we now develop a graphical/combinatorial approach for computing  $\cO_n(\mathbf{z})$ and hence $G_n$ for all $n$. 
We define an \textit{order $n$ Virasoro graph} to be a directed  graph $\mc{G}^{n}$ with $n$ vertices labelled by $z_1,\ldots,z_n$. 
Each $z_{i}$-vertex has degree $\deg(z_i)=0,1$ or $2$.
The degree $1$ vertices can have either unit indegree or outdegree 
whereas the
degree $2$ vertices have both unit indegree and outdegree. 
Thus,  the connected subgraphs of $\mc{G}^{n}$ consist of $r$-cycles, with  $r\ge 1$ degree $2$    vertices,  and chains  with two degree 1 end-vertices with all vertices of  degree 2. 
We regard a single degree 0 vertex as being a degenerate chain. 
\begin{remark}\label{rem:permutation}
The set of non-isomorphic order $n$ Virasoro graphs is in one to one correspondence with the set of partial permutations of the label set $\lbrace 1,\ldots,n\rbrace$. This is described in further detail in \cite{HT1}.  
\end{remark}
We define a genus two weight $W(\mc{G}^{n})$ on $\mc{G}^{n}$ as follows. For each directed edge $\mc{E}_{ij}$  
we define an edge weight
\begin{align}
W(\mc{E}_{ij}) = 
W(
\xy(0,0)*{\cir<3pt>{}}="a"*+!R{z_i\,}; (10,0)*{\cir<3pt>{}}="b"*+!L{\,z_j}; \ar "b";"a";\endxy
)
=
\left\{
  \begin{array}{ll}
    \frac{1}{6}s(z_i) &\mbox{ for } i=j,\\
    \omega(z_{i},z_j) &\mbox{ for } i\neq j.
  \end{array}\right.
  \label{Wij}
\end{align}
Let ${\mathcal C}_{k\ell}$ denote a chain in $\mc{G}^n$ with end-vertices $z_k$ and $z_\ell$
\[
\xy
(-10,0)*[o]=<0.4pt>+{\cir<3pt>{}}="a"*+!D{z_m};
(10,0)*[o]=<0.4pt>+{\cir<3pt>{}}="d"*+!D{z_n};
(-20,0)*[o]=<0.4pt>+{\cir<3pt>{}}="b"*+!R{z_k\,};
(20,0)*[o]=<0.4pt>+{\cir<3pt>{}}="c"*+!L{\,z_\ell};
(0,0)*[o]=<6pt>+{\cdots}="e";
\ar "a";"b"; \ar "d";"c"
\endxy
\]
and assign a chain weight (including the degenerate chain) 
\begin{align}
W(\mc{C}_{k\ell}) 
= W(\xy
(-10,0)*[o]=<0.4pt>+{\cir<3pt>{}}="a";
(10,0)*[o]=<0.4pt>+{\cir<3pt>{}}="d";
(-20,0)*[o]=<0.4pt>+{\cir<3pt>{}}="b"*+!R{z_k\,};
(20,0)*[o]=<0.4pt>+{\cir<3pt>{}}="c"*+!L{\,z_\ell};
(0,0)*[o]=<6pt>+{\cdots}="e";
\ar "a";"b"; \ar "d";"c"
\endxy)
=A(z_k,z_\ell),
  \label{WCij}
\end{align}
where $A(z_k,z_\ell)=\sum_{1\le a\le b\le 2} \nu_{a}(z_k)\nu_{b}(z_\ell)\alpha_{ab}$ 
for free parameters $\alpha_{ab}=\alpha_{ba}$.
Let $K$ be the number of cycles and define a weight for $\mc{G}^n$ by
\begin{align}
W(\mc{G}^n) = \left(\frac{c}{2}\right)^K \prod_{\mathcal{E}_{ij}}W(\mc{E}_{ij})
\prod_{\mc{C}_{k\ell}}W(\mc{C}_{k\ell}),\label{g1_weightofgraph}
\end{align}
where the first product ranges over all the edges and the second product ranges over all the chains of $\mc{G}^n$.
Thus the weight depends on $c$, $\omega(z_i,z_j)$, $s(z_i)$, $\nu_a(z_i)$ and $\alpha_{ab}$. 
We also note that $W$ is multiplicative on the disconnected components of $\mc{G}^n$.

Lastly, define a linear map $\mathcal{L}_{\alpha}$ from $\mathbb{C}[\alpha_{ab}]$, the vector space of complex coefficient polynomials in $\alpha_{ab}$,  to the complex vector space spanned by $\frac{\partial}{\partial \Omega_{ab}}$ derivatives with
\begin{align}
   \mathcal{L}_{\alpha}\left( \alpha_{a_{1}b_{1}}\ldots \alpha_{a_{M}b_{M}}\right) &=
	\frac{\partial^M}{\partial \Omega_{a_{1}b_{1}}\ldots \partial \Omega_{a_{M}b_{M}}}.
   \label{eq:Lmap}
\end{align}
Let $p^n_{KM}$ be the number of inequivalent order $n$ Virasoro graphs containing $K$ cycles and $M$ chains. 
In \cite{HT1} the following  graph generating function is established
\begin{align}
p^{n}(\alpha,\beta)=\sum_{K\ge 0,M\ge 0}p^n_{KM}\alpha^M \beta^K=(-1)^n n! \sum_{i=0}^{n} \frac{(-\alpha)^i}{i!}\binom{-\beta-i}{n-i},
\label{pnKM}
\end{align}
for chain and cycle counting parameters $\alpha$ and $\beta$ respectively.
Thus for $n=1$ we find $p^1(\alpha,\beta) = \alpha+\beta$ corresponding to two inequivalent graphs  with weights
\begin{align*}
W\Big(\xy (0,0)*{\cir<3pt>{}}="a"*+!R{z_1\,};\endxy\Big)= A(z_1,z_1),
\qquad
W\Big(\xy (0,0)*[o]=<0.5pt>+{\cir<3pt>{}}="b"*+!R{z_1\,};
\ar@(ru,lu) "b";"b"
\endxy\Big)= \frac{c}{2}\frac{s(z_1)}{6},
\end{align*}
whose weight sum under the action of $\mc{L}_{\alpha}$ is $\cO_{1}(z_1)$ using \eqref{eq:On1}.

For $n=2$ we have $p^2(\alpha,\beta) = \alpha^2+2\alpha\beta+\beta^2+\beta+2\alpha$ for  7  graphs with weights:
\begin{align*}
&
W\Big(\xy
(0,0)*[o]=<0.4pt>+{\cir<3pt>{}}="a"*+!R{z_1\,}; (5,0)*[o]=<0.4pt>+{\cir<3pt>{}}="b"*+!L{\,z_2};
\endxy\Big)
=A(z_1,z_1)A(z_2,z_2),
\\
&
W\Big(
\xy
(0,0)*[o]=<0.4pt>+{\cir<3pt>{}}="a"*+!R{z_1\,}; (5,0)*[o]=<0.4pt>+{\cir<3pt>{}}="b"*+!L{\,z_2}; \ar@(ru,lu) "a";"a";\endxy\Big)=\frac{c}{2}\frac{s(z_1)}{6}A(z_2,z_2),
\qquad
W\Big(
\xy
(0,0)*[o]=<0.4pt>+{\cir<3pt>{}}="a"*+!R{z_1\,}; (5,0)*[o]=<0.4pt>+{\cir<3pt>{}}="b"*+!L{\,z_2}; \ar@(ru,lu) "b";"b";\endxy\Big)=\frac{c}{2}\frac{s(z_2)}{6}A(z_1,z_1),
\\
&
W\Big(
\xy
(0,0)*[o]=<0.4pt>+{\cir<3pt>{}}="a"*+!R{z_1\,}; (8,0)*[o]=<0.4pt>+{\cir<3pt>{}}="b"*+!L{\,z_2}; \ar@(ru,lu) "a";"a";\ar@(ru,lu) "b";"b";\endxy\Big)=\left(\frac{c}{2}\right)^2\frac{s(z_1)}{6}\frac{s(z_2)}{6},
\qquad
W\Big(
\xy
(0,0)*[o]=<0.4pt>+{\cir<3pt>{}}="a"*+!R{z_1\,}; (8,0)*[o]=<0.4pt>+{\cir<3pt>{}}="b"*+!L{\,z_2}; \ar@/^/ "b";"a";\ar@/^/ "a";"b";\endxy\Big)=\frac{c}{2}\omega(z_1,z_2)^2,
\\
&
W\Big(
\xy
(0,0)*[o]=<0.4pt>+{\cir<3pt>{}}="a"*+!R{z_1\,}; (8,0)*[o]=<0.4pt>+{\cir<3pt>{}}="b"*+!L{\,z_2}; \ar "a";"b";\endxy\Big)=
W\Big(
\xy
(0,0)*[o]=<0.4pt>+{\cir<3pt>{}}="a"*+!R{z_1\,}; (8,0)*[o]=<0.4pt>+{\cir<3pt>{}}="b"*+!L{\,z_2}; \ar "b";"a";\endxy\Big)=A(z_1,z_2),
\end{align*}
whose weight  sum under the action of $\mc{L}_{\alpha}$ 
using \eqref{eq:On2} is 
\begin{align*}
\sum_{\mc{G}^2}\mc{L}_{\alpha}\left(W(\mc{G}^2)\right)&=\cO_2(z_1,z_2).
\end{align*}
These examples illustrate the general result:
\begin{theorem}\label{theor:main}
The order $n$ genus two Virasoro generating function is given by \begin{align*}
G_n(\mathbf{z})=Z_M^{(2)}(\tau_1,\tau_2,\epsilon)^{c}\, \cO_n(\mathbf{z})\Theta_{V}(\tau_1,\tau_2,\epsilon),
\end{align*}
for linear differential operator 
\begin{align}
\cO_n(\mathbf{z}) = \sum_{\mc{G}^n}\mathcal{L}_{\alpha}(W(\mc{G}^n)),
\label{eq:main}
\end{align}
where the sum is taken over all inequivalent order $n$ Virasoro graphs $\mc{G}^n$.
\end{theorem}
\begin{proof}
We prove the result by induction in $n$. 
We have already shown the result holds for $n=1$ and $n=2$ and employ the Ward identity \eqref{eq:Ward} to inductively prove \eqref{eq:main} for $n\ge 2$.

Every inequivalent order $n$ Virasoro graph $\mathcal{G}^{n}$ can be characterized,  according to the nature of the $z_1$ vertex, in terms of following five types:
\begin{enumerate}[(i)]
\item $\deg(z_{1})=0$: $\xy(0,0)*[o]=<0.4pt>+{\cir<3pt>{}}="a"*+!R{z_1\,};\endxy
\ \cdots$
\item $\deg(z_{1})=1$:
$\xy(0,0)*[o]=<0.4pt>+{\cir<3pt>{}}="a"*+!R{z_1\,};(10,0)*[o]=<0.4pt>+{\cir<3pt>{}}="c"*+!L{\,z_a};\ar "c";"a";\endxy 
\cdots$ or 
$\xy(0,0)*[o]=<0.4pt>+{\cir<3pt>{}}="a"*+!R{z_1\,};(10,0)*[o]=<0.4pt>+{\cir<3pt>{}}="c"*+!L{\,z_a};\ar "a";"c";\endxy \cdots$
\item $\deg(z_{1})=2$ where the $z_{1}$-vertex forms a 1-cycle:
$
\xy(0,0)*[o]=<0.4pt>+{\cir<3pt>{}}="a"*+!R{z_1\,};\ar@(dr,ur) "a";"a";\endxy
\cdots
$
\item $\deg(z_{1})=2$ where the $z_{1}$-vertex is an element of a  $2$-cycle: 
$
\xy
(0,0)*[o]=<0.4pt>+{\cir<3pt>{}}="a"*+!R{z_1}; (10,0)*[o]=<0.4pt>+{\cir<3pt>{}}="b"*+!L{z_k}; \ar@/^/ "b";"a";\ar@/^/ "a";"b";\endxy
\cdots
$
\item $\deg(z_{1})=2$ where either the $z_{1}$-vertex is a non end-vertex of a chain or an element of an  $r$-cycle with $r\ge 3$: 
$
\xy
(0,0)*[o]=<0.4pt>+{\cir<3pt>{}}="a"*+!D{z_1};
(-10,0)*[o]=<0.4pt>+{\cir<3pt>{}}="b"*+!R{z_a\,};
(10,0)*[o]=<0.4pt>+{\cir<3pt>{}}="c"*+!L{\,z_b};
(-20,0)*[o]=<6pt>+{\cdots}="d";
(20,0)*[o]=<6pt>+{\cdots}="e";
\ar "a";"b"; \ar "a";"c"
\endxy
$
\end{enumerate}
The Ward identity \eqref{eq:Ward} and \eqref{eq:DZM} imply we may recursively describe $\cO_n(\mathbf{z})$ as follows:
\begin{align}
\cO_n(\mathbf{z})=&\frac{c}{12} s(z_1){\cO}_{n-1}(z_2,\ldots,z_n)
\notag
 \\
&+\Bigg( \nabla_{z_1} 
+ dz_1^2\sum_{k=2}^n \left( \sups{2}\mathcal{P}_{1}(z_1,z_k) \partial_{z_k} + 2\cdot\sups{2}\mathcal{P}_{2}(z_1,z_k) \right)\Bigg){\cO}_{n-1}(z_2,\ldots,z_n) 
\notag
\\
&
+\frac{c}{2}\, \sum_{k=2}^n \sups{2}\mathcal{P}_{4}(z_1,z_k) \, dz_1^2dz_k^2
\,{\cO}_{n-2}(z_2,\ldots,\widehat{z}_k,\ldots, z_n),
\label{eq:WardOn}
\end{align}
We now show how the parts of \eqref{eq:WardOn} relate to  Virasoro graph weights by using induction in $n$. Thus given $\cO_{n-1}$ and $\cO_{n-2}$ satisfy \eqref{eq:main}, we  see that the $\frac{c}{12} s(z_1){\cO}_{n-1} $ term  of \eqref{eq:WardOn} arises from the sum over all $\mathcal{G}^{n}$ graphs of type~(iii). 

Let $\mathcal{G}^{n-1}$ denote an order $n-1$ Virasoro graph labelled by $z_{2},\ldots, z_{n}$
of weight $W(\mc{G}^{n-1})$. This gives a contribution to \eqref{eq:WardOn} of 
\begin{align}
\label{eq:Leib}
\nabla_{z_1}\mc{L}_{\alpha}\left(W(\mc{G}^{n-1})\right)
 &=
 \mc{L}_{\alpha}\left(W(\mc{G}^{n-1})A(z_1,z_1)\right)
+  \mc{L}_{\alpha}\left(\nabla_{z_1}W(\mc{G}^{n-1})\right),
\end{align} 
using the Leibniz rule for $\nabla_x$. 
In particular, all terms of the form $W(\mc{G}^{n-1})A(z_1,z_1)$ arise as weights of  $\mathcal{G}^{n}$ graphs of type~(i). 

Let us examine the contributions that arise from 
$\nabla_{z_1}W(\mc{G}^{n-1})$ in \eqref{eq:Leib} and the remaining  terms
in \eqref{eq:WardOn} and show that these can be expressed in terms of a sum of the weights of graphs of type~(ii), (iv) and (v).
Let $z_k$ be a given vertex in $\mc{G}^{n-1}$ for $k=2,\ldots ,n$. Then, much as before, we can characterize $\mc{G}^{n-1}$ according to (a) $z_k$ is a degree 0 vertex  (b)  $z_k$ is a disconnected vertex of degree 2 (c) $z_k$ is a degree 1 vertex   or (d)  $z_k$ is a degree 2 vertex in a chain or an $r$-cycle for $r\ge 2$. 

\textbf{Case (a).} $\mc{G}^{n-1}$ consists of a  $z_k$ vertex of degree $0$ and an order $n-2$ Virasoro graph $\mc{G}^{n-2}$ (with vertices $z_2,\ldots,\widehat{z}_k,\ldots,z_n$) 
of weight
\begin{align*}
W(\mc{G}^{n-1})=A(z_k,z_k)W(\mc{G}^{n-2}).
\end{align*}
Using \eqref{eq:Dnunu} this  contributes to \eqref{eq:WardOn} the term 
\begin{align*}
\mc{D}_{z_1,z_k}A(z_k,z_k)W(\mc{G}^{n-2})=
2A(z_1,z_k)\omega(z_1,z_k)W(\mc{G}^{n-2}),
\end{align*}
the sum of the weights of two $\mathcal{G}^{n}$ graphs  of type~(ii) where $z_1$ and $z_k$ form a disconnected chain of length $2$.

\textbf{Case (b).} $\mc{G}^{n-1}$ consists of a disconnected degree 2 vertex $z_k$  and an order $n-2$ Virasoro graph $\mc{G}^{n-2}$ of weight
$W(\mc{G}^{n-1})= \frac{c}{12}s(z_k)W(\mc{G}^{n-2})$
which contributes $\frac{c}{12}\mc{D}_{z_1,z_k}(s(z_k))W(\mc{G}^{n-2})$  to \eqref{eq:WardOn}. Summing with the  
$\frac{c}{2} \sups{2}\mathcal{P}_{4}(z_1,z_k)W(\mc{G}^{n-2})$ contribution to  \eqref{eq:WardOn} gives
\begin{align*}
\frac{c}{2}\omega(z_1,z_k)^2W(\mc{G}^{n-2}),
\end{align*}
using \eqref{eq:SDE}, the weight of a $\mathcal{G}^{n}$ graph of type~(iv) where $z_1$ and $z_k$ form a 2-cycle.

\textbf{Case (c).} $z_k$ is an end-vertex  of a chain $\mc{C}_{k\ell}$ so that 
$W(\mc{G}^{n-1})=A(z_k,z_\ell)\omega(z_k,z_m)\ldots$, where $z_k$ is joined to $z_m$ and the ellipsis denotes the factors independent of $z_k$. Using \eqref{eq:omDE}  and \eqref{eq:nuDE} this  contributes terms to \eqref{eq:WardOn} of the form
\begin{align}
&\Bigg(\Big( \nabla_{z_1} 
+ dz_1^2  \Big( \sups{2}\mathcal{P}_{1}(z_1,z_k) \partial_{z_k} + 2\cdot\sups{2}\mathcal{P}_{2}(z_1,z_k) 
\notag
\\
&
+\sups{2}\mathcal{P}_{1}(z_1,z_m) \partial_{z_m} + \sups{2}\mathcal{P}_{2}(z_1,z_m)\Big)\Big)A(z_k,z_\ell)\omega(z_k,z_m)\Bigg)\ldots
\notag
\\
& =A(z_1,z_\ell)\omega(z_1,z_k)\omega(z_k,z_m)\ldots
+A(z_k,z_\ell)\omega(z_k,z_1)\omega(z_1,z_m)\ldots.
\label{eq:chain}
\end{align} 
Note that we have omitted in \eqref{eq:chain} contributions  to \eqref{eq:WardOn} of the form:
\begin{align*}
A(z_k,z_\ell)\omega(z_k,z_m)\left(\nabla_{z_1} +\sups{2}\mathcal{P}_{1}(z_1,z_m) \partial_{z_m} +
\sups{2}\mathcal{P}_{2}(z_1,z_m)\right)(\ldots)
\end{align*}
which contribute to case~(d) for $z_m$.
The first term in \eqref{eq:chain} is the weight of a $\mathcal{G}^{n}$ graph of type~(ii): 
\[
\xy
(-10,0)*[o]=<0.4pt>+{\cir<3pt>{}}="a"*+!R{z_1\,};
(0,0)*[o]=<0.4pt>+{\cir<3pt>{}}="b"*+!D{z_k};
(10,0)*[o]=<0.4pt>+{\cir<3pt>{}}="f"*+!D{z_m};
(35,0)*[o]=<0.4pt>+{\cir<3pt>{}}="c"*+!L{\,z_\ell};
(25,0)*[o]=+{}="d";
(20,0)*[o]=<6pt>+{\cdots} ="e";
\ar "b";"a"; \ar "f";"b";\ar "d";"c"
\endxy
\quad \ldots
\]
 and  the second term is the weight of a graph of type~(v):
\[
\xy
(-10,0)*[o]=<0.4pt>+{\cir<3pt>{}}="a"*+!R{z_k\,};
(0,0)*[o]=<0.4pt>+{\cir<3pt>{}}="b"*+!D{z_1};
(10,0)*[o]=<0.4pt>+{\cir<3pt>{}}="f"*+!D{z_m};
(35,0)*[o]=<0.4pt>+{\cir<3pt>{}}="c"*+!L{\,z_\ell};
(25,0)*[o]=+{}="d";
(20,0)*[o]=<6pt>+{\cdots} ="e";
\ar "b";"a"; \ar "f";"b";\ar "d";"c"
\endxy
\quad\ldots
\]

\textbf{Case (d).} If $\deg(z_k)=2 $ then
$W(\mc{G}^{n-1})=\omega(z_a,z_k)\omega(z_k,z_b)\ldots$, where $z_k$ is joined to $z_a$ and $z_b$ and the ellipsis denotes the factors independent of $z_k$. This  contributes terms to \eqref{eq:WardOn} of the form
\begin{align}
&\Bigg(\Big( \nabla_{z_1} 
+ dz_1^2  \Big(  \sups{2}\mathcal{P}_{1}(z_1,z_k) \partial_{z_k} +\sups{2}\mathcal{P}_{1}(z_1,z_a) \partial_{z_a} +\sups{2}\mathcal{P}_{1}(z_1,z_b) \partial_{z_b} 
\notag
\\
&
+\sups{2}\mathcal{P}_{2}(z_1,z_a)+ 2\cdot\sups{2}\mathcal{P}_{2}(z_1,z_k)  
+  \sups{2}\mathcal{P}_{2}(z_1,z_b)\Big)\Big)\omega(z_a,z_k)\omega(z_k,z_b)\Bigg)\ldots
\notag
\\
& =\omega(z_a,z_1)\omega(z_1,z_k)\omega(z_k,z_b)\ldots
+\omega(z_a,z_k)\omega(z_k,z_1)\omega(z_1,z_b)\ldots
\label{eq:zazb}
\end{align} 
using \eqref{eq:omDE}.
Note that we have omitted in \eqref{eq:omDE} contributions  to \eqref{eq:WardOn} of the form:
\begin{align*}
\omega(z_a,z_k)\omega(z_k,z_a)& \Big(
\nabla_x+\sups{2}\mathcal{P}_{1}(z_1,z_a) \partial_{z_a} +\sups{2}\mathcal{P}_{1}(z_1,z_b) \partial_{z_b}  \\
&
+\sups{2}\mathcal{P}_{2}(z_1,z_a)+  \sups{2}\mathcal{P}_{2}(z_1,z_b)\Big)(\ldots)
\end{align*}
 which contribute to case~(c) and case~(d) for  $z_a$ or $z_b$. 
The two terms in \eqref{eq:zazb} are weights of a $\mathcal{G}^{n}$ graphs of type~(v): 
\[
\cdots
\xy
(-10,0)*[o]=<0.4pt>+{\cir<3pt>{}}="a"*+!R{z_a\,};
(0,0)*[o]=<0.4pt>+{\cir<3pt>{}}="b"*+!D{z_1};
(10,0)*[o]=<0.4pt>+{\cir<3pt>{}}="f"*+!D{z_k};
(20,0)*[o]=<0.4pt>+{\cir<3pt>{}}="c"*+!L{\,z_b\,\cdots};
\ar "b";"a"; \ar "f";"b";\ar "f";"c"
\endxy,
\qquad 
\cdots
\xy
(-10,0)*[o]=<0.4pt>+{\cir<3pt>{}}="a"*+!R{z_a\,};
(0,0)*[o]=<0.4pt>+{\cir<3pt>{}}="b"*+!D{z_k};
(10,0)*[o]=<0.4pt>+{\cir<3pt>{}}="f"*+!D{z_1};
(20,0)*[o]=<0.4pt>+{\cir<3pt>{}}="c"*+!L{\,z_b\,\cdots};
\ar "b";"a"; \ar "f";"b";\ar "f";"c"
\endxy
\]
Thus, altogether, we find that the weights of all  $\mc{G}^n$ graphs of type~(i)-(v) contribute and hence \eqref{eq:main} holds.
\end{proof}

\begin{remark}
\label{rem:Onrem}
$\cO_n(\mathbf{z})$ of \eqref{eq:main} is symmetric in $z_1,\ldots,z_n$ and is expressed in a coordinate free way in terms of $\omega(z_i,z_j)$, $\nu_a(z_i)$, $s(z_i)$ and $\frac{\partial}{\partial \Omega_{ab}}$ for all $\Omega\in\HH_2$  where the \textbf{only} dependence on the original VOA is the central charge  $c$. 
Furthermore,  using Corollaries~\ref{diffeqn_indep} and \ref{Seqn_indep} it follows that $\cO_n(\mathbf{z})$  satisfies the recurrence relation:
\begin{align}
\cO_n(\mathbf{z})=&\left( \nabla_{z_1}+ \frac{c}{12} s(z_1)
+ \sum_{k=2}^n dz_k\left( \Psi(z_1,z_k) \partial_{z_k} + 2\partial_{z_k}\Psi(z_1,z_k) \right)\right){\cO}_{n-1}(z_2,\ldots,z_n) 
\notag
\\
&
+3c\, \sum_{k=2}^n dz_k^3\left(\partial_{z_k}^3\Psi(z_1,z_k) \right)  
\,{\cO}_{n-2}(z_2,\ldots,\widehat{z}_k,\ldots, z_n),
\label{eq:Ongen}
\end{align}
for any choice of coordinates $\mathbf{z}$ and for all $\Omega\in\HH_2$.
\end{remark}
\begin{remark}\label{rem:modules}
Theorem~\ref{theor:main} can be readily generalized for any pair of ordinary $V$-modules $W_1,W_2$ with genus two $n$-point function
\begin{align*}
Z_{W_1,W_2}^{(2)}\left(\omt,z_1;\ldots;\omt,z_n\right)d\mathbf{z}^2&=d\mathbf{z}^2
\sum_{u\in V} Z^{(1)}_{W_1}(Y[\omt,z_1]\ldots Y[\omt,z_n]u;\tau_1) Z^{(1)}_{W_2}(\overline{u};\tau_2)\\
&=\cO_n(\mathbf{z})\Theta_{W_1,W_2}(\tau_1,\tau_2,\epsilon),
\end{align*}
where $
\Theta_{W_1,W_2}(\tau_1,\tau_2,\epsilon)=
Z_M^{(2)}(\tau_1,\tau_2,\epsilon)^{-c}Z_{W_1,W_2}^{(2)}(\tau_1,\tau_2,\epsilon)$. 
Following Conjecture~\ref{conj:ThV} we further conjecture that if $V$ is rational and $C_2$ cofinite, then $\Theta_{W_1,W_2}$ for irreducible $W_1,W_2$ form further components of a weight $c/2$ vector valued Siegel modular form.  
\end{remark}

\section{Analytic and modular transformations}\label{sect:an_mod}
By remark~\ref{rem:Onrem}, we may express $\cO_n(\mathbf{z})$ in any coordinate system on an arbitrary genus two Riemann surface.
In particular,  we can  consider the behaviour of $\cO_n(\mathbf{z})$ under a general analytic transformation:
\begin{proposition}
\label{prop:GenTransf}
Let $z\rightarrow\phi(z)$ be an analytic map then we have
\begin{align*}
\cO_n(\mathbf{z})=\cO_n(z_1,\ldots,\phi(z_i),\ldots,z_n)+\frac{c}{12}\{\phi(z_i),z_i\}dz_{i}^2\, \cO_{n-1}(z_1,\ldots,\widehat{z}_i,\ldots,z_n),
\end{align*}
for $i\in\{1,\ldots,n\}$ and $\{\phi(z),z\}$ is the Schwarzian derivative.
\end{proposition}
\begin{proof}
 Choose $i=1$ wlog by Proposition~\ref{prop_Ggen}.
$\omega(z_1,z_j), \nu_a(z_1)$  are invariant  under an analytic transformation whereas $s(z_1)$ 
 transforms as in \eqref{eq:sphi}.
Such a $s(z_1)$ term only arises in the Virasoro graphs $\mathcal{G}^{n}$ of type~(iii), 
where the $z_{1}$-vertex forms a 1-cycle  of weight $W(\mc{G}^{n})=\frac{c}{12}s(z_1)W(\mc{G}^{n-1})$. Thus the result follows. 
\end{proof}

\medskip
In order to describe the genus two modular properties of $\cO_n(\mathbf{z})$ we 
analyse the modular properties of the $(2,-1)$-form $\Psi(x,y) $ of \eqref{eq:P21form}. 
The genus two modular group $\Sp(4,\Z)$ consists of integral block matrices $\gamma:=\left[
\begin{smallmatrix}
A & B\\
C & D 
\end{smallmatrix}
\right]$ where  $A,B,C,D$ obey:
\begin{align}
&A^TD-C^TB=I,\quad\, AB^T=BA^T,\quad\,  CD^T=DC^T ,
\notag
\\
&A^TC=C^TA, \quad\, B^TD=D^T B,
\label{eq:sp4rels}
\end{align} 
for identity matrix $I$. 
It is convenient to define for $\gamma\in\Sp(4,\Z)$ and $\Omega\in \HH_2$ 
\begin{align*}
M=C\Omega+D,\quad N=(C\Omega+D)^{-1}.
\end{align*}
The holomorphic differentials 
$\boldsymbol{\nu}(x)$, the period matrix $\Omega$, the derivative operator $\nabla_x$, the bidifferential $\omega(x,y)$ and the projective connection $s(x)$ transform under $\gamma\in\Sp(4,\Z)$ as follows \cite{F,Mu1, GT1}
\begin{align}
&\boldsymbol{\nu}^{\gamma}(x)=\boldsymbol{\nu}N,\qquad
\Omega^{\gamma} =(A\Omega +B)N,\qquad
\nabla_{x}^{\gamma} =\nabla_{x},
\label{eq:mod}
\\
&\omega^{\gamma}(x,y) =\omega(x,y)
-\frac{1}{2}\sum_{1\le a\le b\le 2}\left(\nu_a(x)\nu_b(y)+\nu_b(x)\nu_a(y)\right)
\frac{\partial}{\partial \Omega_{ab}}\log\det M,\label{eq:ommod}
\\
&s^{\gamma}(x ) =s(x )
-6\,\nabla_x\log\det M. \label{eq:Smod}
\end{align} 
We now show that for the the $(2,-1)$ bidifferential $\Psi(x,y) $ of \eqref{eq:P21form} 
\begin{theorem}
\label{th:2P1mod}
$\Psi(x,y) $  is $\Sp(4,\Z)$ modular  invariant.
\end{theorem}

In order to prove Theorem~\ref{th:2P1mod} we need two lemmas. The first lemma concerns the second term on the right hand side of \eqref{eq:ommod}:
\begin{lemma}
\label{lem:omsp}
\begin{align}
\frac{1}{2}\sum_{1\le a\le b\le 2}\left(\nu_a(x)\nu_b(y)+\nu_b(x)\nu_a(y)\right)
\frac{\partial}{\partial \Omega_{ab}}\log\det M
=\boldsymbol{\nu}(x)NC\boldsymbol{\nu}^T(y).
\end{align}
\end{lemma}

\noindent 
\begin{proof}  Using the $\Sp(4,\Z)$ relations \eqref{eq:sp4rels} we  find $
N=A^T-C^T \Omega^{\gamma}$ so that 
\begin{align}
(NC)^T=NC.
\label{eq:NCsym}
\end{align}
 The result follows by direct calculation using \eqref{eq:NCsym} where we find
\begin{align*}
\partial_{11}\log\det M&=(M_{22}C_{11}-M_{12}C_{21})/\det M=(NC)_{11},\\
\partial_{22}\log\det M &=(NC)_{22},\\
\partial_{12}\log\det M&=2(NC)_{12}.
\end{align*}
\end{proof}

\begin{lemma}
\label{lem:detsp}
For $ \boldsymbol{\nu}^{\gamma}(x)=\boldsymbol{\nu}N$ of \eqref{eq:mod} we have 
\begin{align}
\begin{vmatrix}
\boldsymbol{\nu}^{\gamma} (x)   \\
 \boldsymbol{\nu}^{\gamma}(y)
\end{vmatrix}
&=
\begin{vmatrix}
\boldsymbol{\nu}(x) \\
\boldsymbol{\nu}(y)
\end{vmatrix}
\det N,
\label{eq:det1}
\\
\begin{vmatrix}
\boldsymbol{\nu}^{\gamma}(y)\\
 \partial_{y} \, \boldsymbol{\nu}^{\gamma}(y)
\end{vmatrix}
&=
\begin{vmatrix}
\boldsymbol{\nu}(y)\\
 \partial_{y} \,\boldsymbol{\nu}(y)
\end{vmatrix}
\det N,
\label{eq:det2}
\\
\begin{vmatrix}
\boldsymbol{\nu}^{\gamma}(y)\\
\nabla^{\gamma}_x \,\boldsymbol{\nu}^{\gamma}(y)
\end{vmatrix}
&=
\begin{vmatrix}
\boldsymbol{\nu}(y)\\
{ \nabla}_x  \,\boldsymbol{\nu}(y)
\end{vmatrix} \det N
+\boldsymbol{\nu}(x)NC\boldsymbol{\nu}^T(y)
\begin{vmatrix}
\boldsymbol{\nu}(x) \\
\boldsymbol{\nu}(y)
\end{vmatrix}
 \det N.
\label{eq:det3}
\end{align}
\end{lemma}

\noindent 
\begin{proof}  
$
\begin{vmatrix}
\boldsymbol{\nu}^{\gamma}(x)   \\
\boldsymbol{\nu}^{\gamma}(y)
\end{vmatrix} 
= 
\begin{vmatrix}
\boldsymbol{\nu}(x)N  \\
\boldsymbol{\nu}(y)N 
\end{vmatrix}
=
\begin{vmatrix}
\boldsymbol{\nu}(x)  \\
\boldsymbol{\nu}(y) 
\end{vmatrix} \det N$ and similarly for \eqref{eq:det2}.
To prove \eqref{eq:det3}, we first note that 
\begin{align*}
\nabla_x M=C\boldsymbol{\nu}(x)^T \boldsymbol{\nu}(x). 
\end{align*}
 Furthermore, 
$(\nabla_x 
N)M 
=-N \nabla_x M$ 
so that 
\begin{align}
\nabla_x N&=-N  (\nabla_x M)N =-NC\boldsymbol{\nu}(x)^T \boldsymbol{\nu}(x) N.
\label{eq:nablaN}
\end{align}
Hence, using \eqref{eq:mod}  and \eqref{eq:NCsym},  we find that 
\begin{align*}
\begin{vmatrix}
\boldsymbol{\nu}^{\gamma}(y)\\
 \nabla^{\gamma}_x \, \boldsymbol{\nu}^{\gamma}(y)
\end{vmatrix} 
=& 
\begin{vmatrix}
\boldsymbol{\nu}(y)N \\
\nabla_x  \left(\boldsymbol{\nu}(y)\right)N +\boldsymbol{\nu}(y)\nabla_x  N
\end{vmatrix}
\\
=& 
\begin{vmatrix}
\boldsymbol{\nu}(y)\\
{ \nabla}_x \, \boldsymbol{\nu}(y)
\end{vmatrix} \det N
+\begin{vmatrix}
\boldsymbol{\nu}(y)N \\
-\boldsymbol{\nu}(y)NC\boldsymbol{\nu}(x)^T  \boldsymbol{\nu}(x)N 
\end{vmatrix}
\\
=&\begin{vmatrix}
\boldsymbol{\nu}(y)\\
{ \nabla}_x  \,\boldsymbol{\nu}(y)
\end{vmatrix} \det N
+\boldsymbol{\nu}(x)NC\boldsymbol{\nu}^T(y) 
\begin{vmatrix}
\boldsymbol{\nu}(x) \\
\boldsymbol{\nu}(y)
\end{vmatrix}
\det N.
\end{align*}  
\end{proof} 
\medskip
\begin{proof}[Proof of Theorem~\ref{th:2P1mod}]
Combining Lemmas~\ref{lem:omsp} and \ref{lem:detsp} we immediately obtain 
\begin{align*}
\omega^\gamma(x,y)\begin{vmatrix}
\boldsymbol{\nu}^\gamma(x) \\
\boldsymbol{\nu}^\gamma(y)
\end{vmatrix}
+\begin{vmatrix}
\boldsymbol{\nu}^\gamma(y)\\
\nabla^\gamma_x \,\boldsymbol{\nu}^\gamma(y)
\end{vmatrix}=\left[
\omega(x,y)\begin{vmatrix}
\boldsymbol{\nu}(x) \\
\boldsymbol{\nu}(y)
\end{vmatrix}
+\begin{vmatrix}
\boldsymbol{\nu}(y)\\
\nabla_x \,\boldsymbol{\nu}(y)
\end{vmatrix}
\right]\det N.
\end{align*}
Thus using \eqref{eq:det2} we find $\Psi(x,y) $ is $\Sp(4,\Z)$ modular  invariant.   
\end{proof}

\begin{corollary}
\label{cor:ODEmod}
The differential equations \eqref{eq:omDE_indep}-\eqref{eq:SDE_indep} are $\Sp(4,\Z)$ invariant.
\end{corollary}
\begin{proof}
Under the action of $\gamma\in \Sp(4,\Z)$,  the change in the left hand side of  \eqref{eq:omDE_indep} is
\begin{align*}
&-\nabla_x \left(\boldsymbol{\nu}(y_1)NC\boldsymbol{\nu}^T(y_2)\right)
-  \sum_{r=1}^2 \partial_{y_r} \left( \Psi(x,y_r) \boldsymbol{\nu}(y_1)NC\boldsymbol{\nu}^T(y_2) \right)dy_r
\\
&=
-\sum_{r=1}^2\omega(x,y_r)\boldsymbol{\nu}(x)NC\boldsymbol{\nu}^T(y_r) 
-\boldsymbol{\nu}(y_1)(\nabla_x  N)C\boldsymbol{\nu}^T(y_2),
\end{align*}
using \eqref{eq:nuDE_indep}.
But  \eqref{eq:NCsym} and \eqref{eq:nablaN}  imply 
\[
-\boldsymbol{\nu}(y_1)(\nabla_x  N)C\boldsymbol{\nu}^T(y_2)
=\left(\boldsymbol{\nu}(x)NC\boldsymbol{\nu}^T(y_1)\right)\left(\boldsymbol{\nu}(x)NC\boldsymbol{\nu}^T(y_2)\right).
\]
Thus the total change  in the left hand side of \eqref{eq:omDE_indep} is 
\[
\omega^\gamma(x,y_1)\omega^\gamma(x,y_2)-\omega(x,y_1)\omega(x,y_2),
\]
as required. A similar method shows that \eqref{eq:nuDE_indep} is modular invariant. Lastly, $\partial_y^k \Psi(x,y)$ is modular invariant so that modular invariance of \eqref{eq:SDE_indep} follows by considering the $y_1\rightarrow y_2$ limit in the above analysis as described in the proof of Corollary~\ref{Seqn_indep}. 
\end{proof} 
\medskip

 Let us now consider the modular properties of the operator $\cO_n(\mathbf{z})$. Following Remark~\ref{rem:Onrem} we know that $\cO_n(\mathbf{z})$ depends on $\omega(x,y)$, $\nu_1(x)$, $\nu_2(x)$, $s(x)$ and $\frac{\partial}{\partial \Omega_{ab}}$  for all $\Omega\in\HH_2$. These terms transform under  $\gamma\in\Sp(4,\Z)$  as in \eqref{eq:mod}-\eqref{eq:Smod} so that
\[
\cO_n(\mathbf{z})\rightarrow
\cO_n^\gamma(\mathbf{z}).
\]
We then find
\begin{theorem}
\label{th:OnMod}
For differentiable $F=F(\Omega)$  we have for  all $\gamma\in \Sp(4,\Z)$ that
\begin{align}
\cO_n^\gamma(\mathbf{z}) \left( \det (M)^{c/2}F\right)=
\det(M)^{c/2}\cO_n(\mathbf{z})  F.
\label{eq:OnMod}
\end{align}
\end{theorem}
\begin{proof}
We prove the result by induction in $n$. The result is trivially true for $n=0$. For $n=1$  we use \eqref{eq:On1}  to find
\begin{align*}
\cO_1^\gamma(z_1) \left( \det (M)^{c/2}F\right)&=
\left(\nabla^\gamma_{z_1}+ \frac{c}{12}s^\gamma(z_{1})\right)\left( \det (M)^{c/2}F\right)\\
&=\det (M)^{c/2}\cO_1(z_1)F,
\end{align*}
using  \eqref{eq:mod} and \eqref{eq:Smod}.
\eqref{eq:Ongen} implies by induction that for $n\ge 2$  
\begin{align*}
& \cO_n^\gamma(\mathbf{z})\left( \det (M)^{c/2}F\right)
\\
 &=
\Bigg( \nabla^\gamma_{z_1} +\frac{c}{12} s^\gamma(z_1)
+ \sum_{k=2}^n dz_k \left( \Psi(z_1,z_k) \partial_{z_k} + 2\partial_{z_k}\Psi(z_1,z_k) \right)\Bigg)\left( \det (M)^{c/2}{\cO}_{n-1}F\right)
\notag
\\
&\quad 
+3c\det (M)^{c/2} \sum_{k=2}^n dz_k^3\left(\partial_{z_k}^3\Psi(z_1,z_k)\right) {\cO}_{n-2}F
\\ &=  \det (M)^{c/2}\cO_n(\mathbf{z})  F,
\end{align*}
using \eqref{eq:mod} and \eqref{eq:Smod} again. Thus the result follows.
\end{proof}

\end{document}